\documentclass[12pt]{amsart}

\usepackage{enumerate, amsmath, amsthm, amsfonts, amssymb, xy,  mathrsfs, graphicx, paralist, fancyvrb}
\usepackage[usenames, dvipsnames]{xcolor}
\usepackage[margin=1in]{geometry} 
\usepackage[bookmarks, colorlinks=true, linkcolor=blue, citecolor=blue, urlcolor=blue]{hyperref}

\input xy
\xyoption{all}

\numberwithin{equation}{subsection}
\newtheorem{theorem}[equation]{Theorem}

\newtheorem{lemma}[equation]{Lemma}

\theoremstyle{definition}
\newtheorem{rmk}[equation]{Remark}
\newenvironment{remark}[1][]{\begin{rmk}[#1] \pushQED{\qed}}{\popQED \end{rmk}}
\newtheorem{eg}[equation]{Example}

\newtheorem{defn}[equation]{Definition}

\newcommand{\rD}{\mathrm{D}}

\newcommand{\rE}{\mathrm{E}}

\newcommand{\bF}{\mathbf{F}}

\newcommand{\bP}{\mathbf{P}}

\newcommand{\bS}{\mathbf{S}}

\newcommand{\rT}{\mathrm{T}}

\newcommand{\bZ}{\mathbf{Z}}

\newcommand{\fg}{\mathfrak{g}}

\newcommand{\bk}{\mathbf{k}}

\renewcommand{\phi}{\varphi}

\newcommand{\ol}[1]{\overline{#1}}

\newcommand{\arxiv}[1]{\href{http://arxiv.org/abs/#1}{{\tt arXiv:#1}}}

\makeatletter
\def\Ddots{\mathinner{\mkern1mu\raise\p@
\vbox{\kern7\p@\hbox{.}}\mkern2mu
\raise4\p@\hbox{.}\mkern2mu\raise7\p@\hbox{.}\mkern1mu}}
\makeatother

\DeclareMathOperator{\rank}{rank}

\DeclareMathOperator{\Sym}{Sym}

\DeclareMathOperator{\Spec}{Spec}

\newcommand{\GL}{\mathbf{GL}}
\newcommand{\SL}{\mathbf{SL}}

\newcommand{\fgl}{\mathfrak{gl}}

\newcommand{\gen}{\mathrm{gen}}

\begin{document}

\title{Schubert varieties and finite free resolutions of length three}

\author{Steven V Sam}
\address{Department of Mathematics, University of California at San Diego, La Jolla, CA  92093}
\email{ssam@ucsd.edu}

\author{Jerzy Weyman}
\address{Department of Mathematics, University of Connecticut, Storrs, CT 06269\newline
  \indent Uniwersytet Jagiello\'nski, Krak\'ow, Poland}
\email{jerzy.weyman@uconn.edu, jerzy.weyman@uj.edu.pl}

\date{April 28, 2020}
\subjclass[2010]{13D02, 13H10, 14M15}

\dedicatory{Dedicated to Laurent Gruson with thanks for his guidance and friendship.}

\maketitle
\begin{abstract}
In this paper we describe the relationship between the finite free resolutions of perfect ideals in split format (for Dynkin formats) and certain intersections of opposite Schubert varieties with the  big cell for homogeneous spaces $G/P$ where $P$ is a maximal parabolic subgroup.

\end{abstract}
\section{Introduction}

Perfect ideals of codimension three have been investigated for a long time. Linkage theory suggests that such ideals might be possible to classify. Indeed, if one applies minimal linkage to a perfect ideal of codimension three, one gets an ideal with a minimal free resolution with the same sum of Betti numbers as the original one, and after a double link one obtains an ideal whose free resolution has modules of the same ranks as for the original ideal.

In the paper \cite{JWm16} the second author constructed the generic rings ${\hat R}_{\gen}$ for resolutions of length three for any format. It turns out that their structure is related to the Kac--Moody Lie algebra of the $T$-shaped graph $\rT_{p,q,r}$ where the triple $(p,q,r)=(r_1+1, r_2-1, r_3+1)$ is constructed from the ranks $r_1, r_2, r_3$ of the differentials $d_1, d_2, d_3$ of our resolution.
We denote the vertices of $\rT_{p,q,r}$ as follows
$$\begin{array}{*{13}{c}} x_{p-1}&-&x_{p-2}&\ldots&x_{1}&-&u&-&y_{1}&\ldots&y_{q-2}&-&y_{q-1}\\
&&&&&&|&&&&&&\\
&&&&&&z_{1}&&&&&&\\
&&&&&&|&&&&&&\\
&&&&&&\vdots&&&&&&\\
&&&&&&z_{r-2}&&&&&&\\
&&&&&&|&&&&&&\\
&&&&&&z_{r-1}&&&&&&
\end{array}$$

This construction suggests that a special role is played by the so-called Dynkin formats for which the graph $\rT_{p,q,r}$ is Dynkin.
In fact, the ring ${\hat R}_{\gen}$ is Noetherian precisely in such cases. One also conjectures (for the case $r_1=1$, i.e., the resolutions of cyclic modules $R/I$) that these are precisely the formats for which every such ideal is licci.

The next idea is to relate the rings ${\hat R}_{\gen}$ to the resolutions of perfect ideals of a given format. The natural idea is to describe an open set $U_{\mathrm{CM}}$ (the Cohen--Macaulay locus) in $\Spec {\hat R}_{\gen}$ where the dual of the generic resolution ${\bf F}^{\gen}_\bullet$ is acyclic. In \cite{CVW-3} and \cite{KHLJW18}  a precise conjecture on the form of the subset $U_{\mathrm{CM}}$ for Dynkin formats is given.
This conjecture is based on the observation that for Dynkin formats (except for the format $(1,n,n,1)$ with $n$ odd), the three top graded components of three critical representations in ${\hat R}_{\gen}$ can be thought of as the differentials in another complex ${\bf F}^{\mathrm{top}}_\bullet$ of the format $(f_0, f_1, f_2, f_3)$ over ${\hat R}_{\gen}$. The conjecture then says that the open set $U_{\mathrm{CM}}$ is equal to the set $U_{\mathrm{split}}$ where the complex ${\bf F}^{\mathrm{top}}_\bullet$ is split exact.

The next task is to find the precise form of the general ideal resolved by a resolution corresponding to a point in $U_{\mathrm{split}}$.
This problem is resolved in \cite{CVW-3} and in \cite{CJKW18} for the formats of type $\rD_n$ and $\rE_6$. It is also proved in \cite{CVW-3} that for the formats $(1,n,n,1)$ with $n$ even, and $(1,4,n,n-3)$, we have the equality $U_{\mathrm{CM}}=U_{\mathrm{split}}$.

The method employed there consists of identifying the generators of ${\hat R}_{\gen}$ occurring in some critical representations in ${\hat R}_{\gen}$ and calculating these structure theorems for split exact resolutions using the defect variables. The structure theorems corresponding to the top graded components then give the differentials in the complex which is a resolution of a perfect ideal of codimension three which corresponds to a general point in $U_{\mathrm{CM}}$. For the remaining types $\rE_7$ and $\rE_8$ the analogous method is very difficult to employ because the representations $W(d_3)$, $W(d_2)$, $W(a_2)$ are too complex.

In the present paper we try another approach. We try to construct the relevant resolutions of Dynkin formats by geometric means. A very interesting geometric pattern appears. We prove that for all the Dynkin types there is a certain  Schubert variety $\Omega_\lambda$ of codimension 3  in a homogeneous space $G(\rT_{p,q,r})/P_{x_1}$ ($G(\rT_{p,q,r})$ denotes the simply-connected simple group corresponding to the graph $\rT_{p,q,r}$, $P_{x_1}$ is a  maximal parabolic subgroup corresponding to the vertex $x_1$ in the above notation), such that the defining ideal of the intersection $Y_\lambda$ of $\Omega_\lambda$ with the opposite big open cell $Y$ has exactly the resolution of the format corresponding to the triple $(p,q,r)$. We expect the singularity of $Y_\lambda$ to be the singularity of the general point in $U_{\mathrm{split}}$ (this is established for the types $\rD_n$ and $\rE_6$).

This proves the inclusion $U_{\mathrm{split}}\subseteq U_{\mathrm{CM}}$ and shows that the expected general form of the finite free resolutions of Dynkin formats show a definite pattern. The pattern conjecturally extends to the resolutions of cyclic modules of arbitrary format. This will be taken up in a future paper.

\subsection*{Acknowledgements}
SS was partially supported by the NSF grant DMS-1849173. 
JW was partially supported by the NSF grant DMS 1802067 and by the grant from Narodowa Agencja Wymiany Akademickiej NAWA in Poland.

This collaboration started during the August 2019 workshop organized by the first author at UCSD.
We thank Shrawan Kumar for patiently answering several questions on Schubert varieties. 
We also benefited from many discussions with Oana Veliche, Lars Christensen, Ela Celikbas, Jai Laxmi, Witold Kra\'skiewicz and Kyu-Hwan Lee.

\section{General setup}

Throughout we work over a field $\bk$. (Our constructions will be insensitive to the choice of field and in fact will be defined over $\bZ$, but we keep a field to avoid making extra remarks throughout.)

Our goal is to describe the general complex  from the open set $U_{\mathrm{split}}$ of $\Spec ({\hat R}_{\gen})$ for a Dynkin format.
This seems to be related to Schubert varieties in homogeneous spaces for groups of type $\rT_{p,q,r}$ and goes as follows.
Consider the Dynkin formats of free resolutions of length 3 of cyclic modules, i.e., assume the rank is $r_1 = 1$.
This means that the corresponding graph $\rT_{p,q,r}$ is $\rT_{2,q,r}$.
We also assume this graph is of Dynkin type.

We fix root data for our group $G$ and look at the homogeneous space $G/P$ where $P=P_{x_1}$ is the maximal parabolic subgroup corresponding to the vertex on the arm of length $1$ corresponding to $p$. We let $B$ and $B^-$ denote the Borel subgroup and opposite Borel subgroup, respectively.
We use basic facts on Schubert varieties, described for example in \cite[\S 3.3]{LR08}. Our notation is for opposite Schubert varieties, i.e., an element of length $\ell$ indexes a variety of codimension $\ell$.
The opposite Schubert varieties in $G/P$ are described by the cosets $W/W_P$ where $W$ is the Weyl group corresponding to $\rT_{p,q,r}$ and $W_P$ is the Weyl group corresponding to smaller root system we get from $\rT_{p,q,r}$ by removing $x_1$. In our case we can identify $W/W_P$ with the $W$-orbit of the fundamental weight $\omega_{x_1}$, and we now describe how this works.

In what follows, we represent weights as a sum of fundamental weights, and represent these by labeling the vertices of $\rT_{p,q,r}$ by the corresponding coefficients. We deal with the Weyl group action on the weight
\[
  \sigma_0=\omega_{x_1}=\begin{matrix}
1&0&0&0\cdots&0\\&0&&&&\\
&\vdots&&&&\\
&0&&&&
\end{matrix}.
\]
Acting by elements of length $\ell$ in $W$ of minimal length in their coset, we get opposite Schubert varieties of codimension $\ell$.
Thus for length $0$ we get $\sigma_0$. For length $1$ we get
$$
\sigma_1=
\begin{matrix}
-1&1&0&0\cdots&0\\&0&&&&\\
&\vdots&&&&\\
&0&&&&
\end{matrix}
$$
and for length $2$
$$\sigma_2=\begin{matrix}
0&-1&1&0\cdots&0\\&1&&&&\\
&\vdots&&&&\\
&0&&&&
\end{matrix}.
$$
Finally for length $3$ we get two possibilities
$$\sigma_3=\begin{matrix}
0&0&1&0&\cdots&0\\&-1&&&&&\\
&1&&&&\\
&\vdots&&&&\\
&0&&&&
\end{matrix}, \qquad
\sigma'_3=\begin{matrix}
0&0&-1&1&\cdots&0\\&1&&&&\\
&0\\
&\vdots&&&&\\
&0&&&&
\end{matrix}.
$$

This picture says that in $G(\rT_{2,q,r})/P_{x_1}$ there is one opposite Schubert variety of codimension $1$, one of codimension $2$, and two of codimension $3$. Define $\Omega^w_P={\ol{B^-wP/P}}$ in $G/P$. We consider the opposite Schubert varieties $\Omega^{\sigma_3}_P$ and $\Omega^{\sigma'_3}_P$ and their intersections $Y^w_P$ with the  big cell $Y$. The cell $Y$ is the unique open $B$-orbit. The varieties $Y_{\sigma_3}, Y_{\sigma'_3}$ are known to be normal Cohen--Macaulay subvarieties of codimension $3$ (see \cite[Proposition 3.4]{KS09}). 

But the same picture represents the extremal Pl\"ucker coordinates of $G(\rT_{2,q,r})/P_{x_1}$ (see \cite[\S 3.3.1]{LR08}) embedded in $\bP(V_{x_1})$ where $V_\lambda$ denotes the irreducible highest weight module of highest weight $\lambda$. Each element $w \in W/W_P$ determines a unique linear function $p_w$ (defined up to scalar) on $\bP(V_{x_1})$ and the defining equations of $Y_{\sigma_3}$ and $Y_{\sigma'_3}$ are the Pl\"ucker coordinates $p_\tau$ such that $\tau$ is not $\le$ than $\sigma_3$ (resp. $\sigma'_3$) in the Bruhat order. This follows from the identification of the linear functions on $\Omega^w_P$ with Demazure modules (see \cite[\S\S 3.3.1, 3.3.3]{LR08}).
In this setup, $Y$ is given by the condition $p_{{\rm top}} \ne 0$ where $p_{{\rm top}}$ is the extremal Schubert coordinate corresponding to the coset of maximal length $\ell = \dim G(\rT_{2,q,r})/P_{x_1}$.

In what follows, working with the parabolic $P_{x_1}$ we give two interpretations of the varieties $Y_{\sigma_3}$, $Y_{\sigma'_3}$: one in terms of Pl\"ucker coordinates, the other in equivariant form related to the grading on $\fg(\rE_m)$ ($m=6,7,8$) corresponding to the simple root $\alpha_{x_1}$. The open cell $Y$ is identified with the unipotent radical of the opposite parabolic subgroup. The functions on this space can be identified with the positive portion of the Lie algebra $\fg(\rE_m)$.

\section{Type $\rE_6$}

The Pl\"ucker coordinates vanishing on $Y_{\sigma_3}$, where
$$\sigma_3 =\begin{matrix}
0&0&1&0\\&-1&&\\
&1&&
\end{matrix}$$
 are the ones corresponding to the elements
$$\begin{matrix}1&0&0&0\\
&0&&\\
&0&&
\end{matrix}, \qquad
\begin{matrix}-1&1&0&0\\
&0&&\\
&0&&
\end{matrix}, \qquad
\begin{matrix}0&-1&1&0\\
&1&&\\
&0&&
\end{matrix}, \qquad
\begin{matrix}0&0&-1&1\\
&1&&\\
&0&&
\end{matrix}, \qquad
\begin{matrix}0&0&0&-1\\
&1&&\\
&0&&
\end{matrix}.$$

There are also 5 Pl\"ucker coordinates vanishing on $Y_{\sigma'_3}$, symmetrically with respect to the two arms of length 3. The intersection of both ideals is a complete intersection given by the three Pl\"ucker coordinates

$$\begin{matrix}1&0&0&0\\
&0&&\\
&0&&
\end{matrix}, \qquad
\begin{matrix}-1&1&0&0\\
&0&&\\
&0&&
\end{matrix}, \qquad
\begin{matrix}0&-1&1&0\\
&1&&\\
&0&&
\end{matrix}$$
that are common in both ideals.

By construction, all opposite Schubert varieties are invariant under the action of the opposite Borel subgroup $B^-$. Generally, there is a larger parabolic subgroup that leaves it invariant. This is explained for Schubert varieties in \cite[Proposition 1.4]{LMS}, but we can use the involution $w \mapsto w w_0^P$ where $w_0^P$ is the minimal length coset representative of the maximal element in $W/W_P$ to get the corresponding statement for opposite Schubert varieties. Namely, the opposite Schubert variety $\Omega^w$ is closed under the root subgroups labeled by a simple root $\alpha$ if $\ell(s_\alpha w) > \ell(w)$.

We see that the ideal of $\Omega^{\sigma_3}_P$ is invariant under the action of the root subgroups corresponding to $\alpha_1, \alpha_4, \alpha_5, \alpha_6$. The Levi subgroup of the corresponding parabolic subgroup is isomorphic to $\GL_2 \times \GL_4$. Analogously, $\Omega^{\sigma_3'}_P$ is invariant under the action of the root subgroups corresponding to $\alpha_1, \alpha_3, \alpha_4, \alpha_6$ and the corresponding Levi subgroup is isomorphic to $\GL_4 \times \GL_2$. The first equation corresponding to $\sigma_0$ generates an ideal which is invariant under $\alpha_1, \alpha_3, \alpha_4, \alpha_5, \alpha_6$, and the corresponding Levi subgroup is isomorphic to $\GL_6$.
 
The equivariant picture is as follows. First, $G/P$ is embedded in the fundamental representation $\bP(V_{\omega_2})$ which is the adjoint representation. Consider the root space decomposition of $\fg(\rE_6)$ and coarsen it to a $\bZ$-grading by considering the coefficient of $\alpha_1$. Then the grading ranges from $-2$ to $2$ and
\[
  \fg_0 = \fgl_6,\qquad \fg_1=\bigwedge^3 \bk^6, \qquad \fg_2 = \bigwedge^6\bk^6.
\]
To get the big cell, we set $p_{\rm top}=1$ where $p_{\rm top}$ spans $\fg_{2}$. The ideal of the embedding of $G/P$ into the adjoint representation is generated by quadratic polynomials, which are homogeneous with respect to the $\alpha_1$-grading. In particular, given any relation that involves $p_{\rm top} q$ for some homogeneous element $q$, we can express $q$ as a sum of product of other elements of smaller $\alpha_1$-degree. In particular, a homogeneous element of $\alpha_1$-degree $d$ restricts to a homogeneous polynomial of degree $2-d$ on $Y$.

The generators for the functions on $Y$ can be identified with $\bigwedge^3 \bk^6 \oplus \bigwedge^6 \bk^6$ where the first space has degree 1 and the second space has degree 2 (and is a subspace of $\fg_0$). Let $A$ denote the coordinate ring of $Y$.

\begin{theorem}
The minimal free resolution of $Y_{\sigma_3}$ is of the form:
\[
  0\rightarrow A^2(-7)\rightarrow A^6(-5)\rightarrow A(-4)\oplus A^4(-3)\rightarrow A.
  \]
  The minimal free resolution of $Y_{\sigma'_3}$ has the same form. Furthermore, the ideals of these varieties are linked by a complete intersection of length $3$.
  The sum of the two ideals has a resolution of format
\[
  0\rightarrow A(-10)\rightarrow A^6(-7)\oplus A(-6)\rightarrow A^{12}(-5)\rightarrow A(-4)\oplus A^6(-3)\rightarrow A.
\]
\end{theorem}

The numerator of the reduced Hilbert series of both $\bk[Y_{\sigma_3}]$ and $\bk[Y_{\sigma'_3}]$ is
\[
  2t^{4}+6t^{3}+6t^{2}+3t+1
\]
and they both  have degree 18. The numerator of the reduced Hilbert series of $\bk[Y_{\sigma_3} \cap Y_{\sigma'_3}]$ is
\[
  t^{6}+4t^{5}+10t^{4}+14t^{3}+10t^{2}+4t+1
\]
and its degree is 44.

\begin{proof}
  We know that the resolution of $Y_{\sigma_3}$ is of the form
  \[
    0 \to \bF_3 \to \bF_2 \to A(-4) \oplus A(-3)^4 \to A. 
  \]
Consider the ideals $I_{\sigma_3}$ and $I_{\sigma'_3}$ in $A$ and the ideal $(\alpha)=(p_{\rm id}, p_{s_{x_1}}, p_{s_u s_{x_1}})$ generated by the first three Pl\"ucker coordinates. This ideal is radical by \cite[Corollary 2.3.3]{brion-kumar}. It follows that $I_{\sigma_3}$ and $I_{\sigma'_3}$ are linked via $(\alpha)$, as this is true set-theoretically and both ideals 
$(\alpha):I_{\sigma_3}$ and $(\alpha):I_{\sigma'_3}$ are radical. Now using \cite[\S 1]{BU90} we see that the ideals  $I_{\sigma_3}$ and $I_{\sigma'_3}$ are linked via $\alpha$, which is a complete intersection generated by polynomials of degrees $3,3,4$.

By the theory of linkage (see for example \cite[\S 5]{DABDEs77}) the dual of the mapping cone of the resolution of the complete intersection to the resolution of $Y_{\sigma_3}$ gives a resolution for $Y_{\sigma'_3}$ (up to a grading shift). This has the form
  \[
    0 \to A(-10) \to \begin{array}{c} A(-10) \oplus\\ A(-7)^4 \oplus\\ A(-6)\end{array}
    \to \begin{array}{c} A(-7)^2 \oplus\\ A(-6) \oplus\\          \bF_2^*(-10) \end{array}
    \to  
    \begin{array}{c} \bF_3^*(-10) \oplus\\ A(-4) \oplus\\ A(-3)^2  \end{array}
    \to A
  \]
  The last few terms cancel since they identify generators and relations from the original two resolutions, and the result is
  \[
    0 \to A(-7)^2 \to \bF_2^*(-10) \to
    \begin{array}{c} \bF_3^*(-10) \oplus\\ A(-4) \oplus\\ A(-3)^2 \end{array}
    \to A
  \]                     
We know that the ideal of $Y_{\sigma'_3}$ is minimally generated by 4 cubics and 1 quartic, and the differential $\bF_2^* \to \bF_3^*$ is minimal, so we conclude that $\bF_3 = A(-7)^2$. This tells us that $\rank \bF_2 = 6$. Let $d_1\le \cdots \le d_6$ be the degrees of its generators. Since $Y_{\sigma_3}$ has codimension 3, the alternating sum
\[
1 - 4t^3 - t^4 + \sum_{i=1}^6 t^{d_i} - 2t^7
\]
is divisible by $(1-t)^3$, which translates to its first two derivatives being 0 at $t=1$:
\begin{align*}
  -12 - 4 + \sum_{i=1}^6 d_i - 14  = 
  -24 - 12 + \sum_{i=1}^6 d_i(d_i-1) - 84 = 0,
\end{align*}
or equivalently,
\begin{align*}
  \sum_{i=1}^6 d_i = 30, \qquad \sum_{i=1}^6 d_i^2 = 150.
\end{align*}
The plane $\sum_{i=1}^6 d_i=30$ is tangent to the sphere defined by the second equation at the point $(5,5,5,5,5,5)$, so this point is the unique solution to both equations.

By \cite[Proposition 1.3]{BU90}, the ideal $I_{\sigma_3}+I_{\sigma'_3}$ is Gorenstein. This is the defining ideal of the Schubert variety $Y_{\sigma_4}$ of codimension $4$ which is contained both in $Y_{\sigma_3}$ and $Y_{\sigma'_3}$. To get the form of its resolution we again appeal to \cite[Proposition 1.3]{BU90} which implies that this resolution is a mapping cone of the map of complexes covering an injection of the canonical module of the coordinate ring of $Y_{\sigma_3}$ (resp. $Y_{\sigma'_3}$) into the resolution of the coordinate ring of $Y_{\sigma_3}$ (resp. $Y_{\sigma'_3}$).
\end{proof}

\subsection{Linear section}

A certain linear section of the first minimal free resolution is described in \cite{CJKW18}. The description of the generators given there is a little different, as the present pattern was not observed at that time. 

Now we consider the restriction of the ideals to $\Sym(\bigwedge^3 \bk^6)$.

\begin{lemma}\label{lem:identificatione6}
  The restriction of $p_{\sigma_3}$ to $\bigwedge^3 \bk^6$ is the unique (up to scalar) $\SL_6$-invariant $\Delta$ of degree $4$ on $\bigwedge^3 \bk^6$.
\end{lemma}

The invariant $\Delta$ can be described as a hyperdiscriminant of the representation $\bigwedge^3 \bk^6$, i.e., the equation of the hypersurface projectively dual to the Grassmannian (see for example \cite[\S 9C]{JWm03}).

\begin{proof}
$p_{\sigma_3}$ is a lowest weight vector and hence has $\alpha_1$-degree $-2$. By what we said above, it becomes a quartic function upon restriction to $Y$. Next, we also said above that the ideal generated by $p_{\sigma_3}$ is invariant under $\GL_6$. This implies that $p_{\sigma_3}$ itself is invariant under the subgroup $\SL_6$. Next, we need to show that this restriction is nonzero. If not, then the quartic in $A$ is $h^2$, where $h$ is the generator of $\bigwedge^6 \bk^6$, which is non-reduced.
\end{proof}

As explained above, $Y_{\sigma_3}$ is invariant under the Levi subgroup $\GL_2 \times \GL_4$, so we choose a decomposition $\bk^6=F\oplus G$ where $F=\bk^4$, $G=\bk^2$ are spanned by the first 4, respectively last 2, coordinate vectors.

 So our restricted Schubert variety is in the affine space $\Spec(A')$ where
 \[
   A'=\Sym(\fg_1) = \Sym(\bigwedge^3 F\oplus\bigwedge^2F\otimes G\oplus F\otimes\bigwedge^2 G).
   \]
 
\begin{lemma} \label{lem:E6-eqn}
  The five defining equations of the restriction of $Y_{\sigma_3}$ to $\Spec(A')$ are
 $\Delta$, and $\frac{\partial\Delta}{\partial y_i}$, where $\lbrace y_i\rbrace$ is a basis of $(\bigwedge^3 F)^*$.
\end{lemma}

\begin{proof}
  We have already explained the appearance of $\Delta$. The remaining 4 functions have $\alpha$-degree $-1$ and hence restrict to cubic functions. By equivariance, these functions are closed under the action of $\GL(F) \times \GL(G)$ and the upper triangular block $G^*\otimes F$, so must span the space $\bigwedge^3 F$. Since the invariant $\Delta$ is unique, this space is spanned by its partial derivatives with respect to $(\bigwedge^3 F)^*$.
\end{proof}
 
 The five equations defining the ideal of $Y_{\sigma'_3}$ restricted to $\Spec(A')$ in this language are similar, we just decompose $\bk^6=F'\oplus G'$ with $\dim F'=2$,
 $\dim G'=4$.

 \begin{lemma}
   The five defining equations of the restriction of $Y_{\sigma'_3}$ to $\Spec(A')$ are  $\Delta$ and $\frac{\partial\Delta}{\partial z_i}$, where $\lbrace z_i\rbrace$ is a basis of the dual of $\bigwedge^2 F'\otimes G'$.
\end{lemma}

The proof is similar to the proof of Lemma~\ref{lem:E6-eqn}.

\section{Type $\rE_7$}

Much of the setup here follows that of the previous section. Rather than re-explain all of the details, we will just list the results of the calculations.

As before, we coarsen the root decomposition of $\fg(\rE_7)$ by considering the coefficient of $\alpha_2$. The non-negative components of the resulting $\bZ$-graded decomposition are as follows:
\[
  \fg_0 = \fgl_7(\bk), \qquad \fg_1 = \bigwedge^3 \bk^7, \qquad \fg_2 = \bigwedge^6 \bk^7.
\]
We let $A$ denote the coordinate ring of the big cell. We can identify its linear functions with the positive portion of this decomposition.

In this case, $V_{x_1}$ is not the adjoint representation, so we need its $\bZ$-graded decomposition as well (see \cite[\S 5.3.3]{KHLJW18}). We do not repeat it here, but just note that the grading goes from 0 to 7 and the top and bottom pieces are 1-dimensional. As before, setting the element in top degree to be nonzero cuts out the big cell.

The ideal of $Y_{\sigma_3}$ is generated by $6$ elements which are Pl\"ucker coordinates corresponding to the elements
$$\begin{matrix}1&0&0&0&0\\
&0&&&\\
&0&&&
\end{matrix}, \qquad
\begin{matrix}-1&1&0&0&0\\
&0&&&\\
&0&&&
\end{matrix}, \qquad
\begin{matrix}0&-1&1&0&0\\
&1&&&\\
&0&&&
\end{matrix},$$
$$
\begin{matrix}0&0&-1&1&0\\
&1&&&\\
&0&&&
\end{matrix}, \qquad
\begin{matrix}0&0&0&-1&1\\
&1&&&\\
&0&&&
\end{matrix}, \qquad
\begin{matrix}0&0&0&0&-1\\
&1&&&\\
&0&&&
\end{matrix}$$
The ideal of $Y_{\sigma'_3}$ is generated by five Pl\"ucker coordinates corresponding to the weights
$$\begin{matrix}1&0&0&0&0\\
&0&&&\\
&0&&&
\end{matrix}, \qquad
\begin{matrix}-1&1&0&0&0\\
&0&&&\\
&0&&&
\end{matrix}, \qquad
\begin{matrix}0&-1&1&0&0\\
&1&&&\\
&0&&&
\end{matrix},$$
$$
\begin{matrix}0&0&1&0&0\\
&-1&&&\\
&1&&&
\end{matrix}, \qquad
\begin{matrix}0&0&1&0&0\\
&0&&&\\
&-1&&&
\end{matrix}$$

\begin{theorem}
  The minimal free resolution of $Y_{\sigma_3}$ is of the form:
\[
0\rightarrow A^2(-13)\rightarrow A^7(-9)\rightarrow A(-7)\oplus A^5(-6)\rightarrow A.
  \]
  The minimal free resolution of $Y_{\sigma'_3}$ is of the form:
  \[
    0\rightarrow A^3(-13)\rightarrow A^7(-10)\rightarrow A(-7)\oplus A^4(-6)\rightarrow A
  \]
  Furthermore, the ideals of these varieties are linked by a complete intersection of length $3$.

The resolution for the sum of the two ideals has the form
\[
  0\rightarrow A(-19)\rightarrow A^7(-13)\oplus A(-12)\rightarrow A^7(-10)\oplus A^7(-9)\rightarrow A(-7)\oplus A^7(-6)\rightarrow A.
\]
\end{theorem}

The numerator of the reduced Hilbert series of $\bk[Y_{\sigma_3}]$ is
\[
  2t^{10}+6t^9+12t^8+20t^7+23t^6+21t^5+15t^4+10t^3+6t^2+3t+1
\]
and hence it has degree 119. The numerator of the reduced Hilbert series of $\bk[Y_{\sigma'_3}]$ is
\[
  3t^{10}+9t^9+18t^8+23t^7+24t^6+21t^5+15t^4+10t^3+6t^2+3t+1
\]
and hence it has degree 133. The numerator of the reduced Hilbert series of $\bk[Y_{\sigma_3} \cap Y_{\sigma'_3}]$ is
\[
  t^{15}+4t^{14}+10t^{13}+20t^{12}+35t^{11}+56t^{10}+77t^{9}+
      91t^{8}+91t^{7}+77t^{6}+56t^{5}+35t^{4}+20t^{3}+10t^{2}+4
      t+1
    \]
    and it has degree 588.

\begin{remark}
We have verified with Macaulay2 that the restriction of $Y_{\sigma_3}$ and $Y_{\sigma'_3}$ to $\bigwedge^3 \bk^7$ continue to have codimension 3 (when ${\rm char}(\bk)=0$). We do not yet have a conceptual proof of this, but we will describe how these ideals look assuming this fact.

There is an invariant $\Delta$ of degree $7$ on $\fg_1$, which can be described as a hyperdiscriminant of the representation $\bigwedge^3 \bk^7$, i.e., the equation of the hypersurface projectively dual to the Grassmannian (see for example  \cite[\S 9C]{JWm03}).

To describe the restriction of $Y_{\sigma_3}$, we write $\bk^7=F\oplus G$ with $\dim F=5$, $\dim G=2$.
We consider the ring
\[
  A'=\Sym(\fg_1) = \Sym(\bigwedge^3F\oplus\bigwedge^2 F\otimes G\oplus F\otimes\bigwedge^2G).
  \]
and the equations of $Y_{\sigma_3}$ restrict to polynomials $\Delta$, $\partial\Delta /\partial{y_i}$, where $\lbrace y_i\rbrace$ is a basis of the dual of $F\otimes\bigwedge^2G$. The argument is similar to the previous section; the key point is why we get a nonzero multiple of $\Delta$: there are no degree 7 invariants on $\Sym(\bigwedge^3 \bk^7 \oplus \bigwedge^6 \bk^7)$ other than $\Delta$.

We get a resolution
$$0\rightarrow A'^2(-13)\rightarrow A'^7(-9)\rightarrow A'(-7)\oplus A'^5(-6)\rightarrow A'.$$
The equivariant form of this resolution is as follows (Here $\bS_\lambda$ denotes a Schur functor. This is valid in characteristic 0, and more care is needed in general to describe the correct representations.):
\[
  0\rightarrow \bS_{6,6,6,6,6}F\otimes \bS_{5,4}G\otimes A'(-13)\to
  \left(\begin{array}{c}
    \bS_{4,4,4,4,4}F\otimes \bS_{4,3}G \oplus\\
    \bS_{5,4,4,4,4}F\otimes \bS_{3,3}G
  \end{array} \right)
  \otimes A'(-9)\rightarrow
  \]
  $$\rightarrow \begin{array}{c}
                  \bS_{3,3,3,3,3}F\otimes \bS_{3,3}G\otimes A'(-7)\oplus \\
                  \bS_{3,3,3,3,2}F\otimes  \bS_{2,2}G\otimes A'(-6)
                \end{array}
                \rightarrow A'.$$

Now we describe the restriction of $Y_{\sigma'_3}$.
We decompose $\bk^7=F'\oplus G'$ where $\dim F'=3$, $\dim G'=4$. We get
$$A''=\Sym(\fg_1)=\Sym(\bigwedge^3 F'\oplus \bigwedge^2F'\otimes G'\oplus F'\otimes\bigwedge^2G'\oplus\bigwedge^3G').$$

We get the ideal of $Y_{\sigma'_3}$ generated by five elements 
$\Delta$, $\partial\Delta /\partial{z_i}$, where $\lbrace z_i\rbrace$ is a basis of the dual of $\bigwedge^3G'$.
This gives a resolution of the format
\[
  0\rightarrow A''^3(-13)\rightarrow A''^7(-10)\rightarrow A''(-7)\oplus A''^4(-6)\rightarrow A''.
\]

The equivariant form of this resolution is (in characteristic 0):
\begin{align*}
  0\rightarrow \bS_{7,6,6}F'\otimes \bS_{5,5,5,5}G' \otimes A''(-13)\rightarrow
\left(  \begin{array}{c}
    \bS_{5,5,4}F'\otimes \bS_{4,4,4,4}G' \oplus\\
    \bS_{5,5,5}F'\otimes \bS_{4,4,4,3}G'
  \end{array} \right)
  \otimes A''(-10)\rightarrow\\
  \rightarrow
  \begin{array}{c}
    \bS_{3,3,3}F'\otimes \bS_{3,3,3,3}G'\otimes A''(-7)\oplus\\
    \bS_{3,3,3}F'\otimes \bS_{3,2,2,2}G' \otimes A''(-6)
  \end{array}
  \rightarrow A''. 
\end{align*}
\end{remark}

\section{Type $\rE_8$}

Much of the setup here follows that of type $\rE_6$. Rather than re-explain all of the details, we will just list the results of the calculations.

The non-negative portion of the $\bZ$-graded decomposition of $\fg(\rE_8)$ with respect to $\alpha_2$ looks like
\[
  \fg_0 = \fgl_8(\bk), \qquad \fg_1 = \bigwedge^3 \bk^8, \qquad \fg_2 = \bigwedge^6 \bk^8, \qquad \fg_3 = \bigwedge^8 \bk^8 \otimes \bk^8.
\]
We let $A$ denote the coordinate ring of the big cell. We can identify its linear functions with the positive portion of this decomposition.

In this case, $V_{x_1}$ is not the adjoint representation, so we need its $\bZ$-graded decomposition as well (see \cite[\S 6.3.3]{KHLJW18}). We do not repeat it here, but just note that the grading goes from 0 to 16 and the top and bottom pieces are 1-dimensional. As before, setting the element in top degree to be nonzero cuts out the  big cell.

The ideal of $Y_{\sigma_3}$ is generated by $7$ elements which are the Pl\"ucker coordinates corresponding to the elements
$$\begin{matrix}1&0&0&0&0&0\\
&0&&&\\
&0&&&
\end{matrix},\qquad
\begin{matrix}-1&1&0&0&0&0\\
&0&&&&\\
&0&&&&
\end{matrix},\qquad
\begin{matrix}0&-1&1&0&0&0\\
&1&&&&\\
&0&&&&
\end{matrix},$$
$$\begin{matrix}0&0&-1&1&0&0\\
&1&&&&\\
&0&&&&
\end{matrix}, \qquad
\begin{matrix}0&0&0&-1&1&0\\
&1&&&&\\
&0&&&&
\end{matrix},\qquad
\begin{matrix}0&0&0&0&-1&1\\
&1&&&&\\
&0&&&&
\end{matrix},\qquad
\begin{matrix}0&0&0&0&0&-1\\
&1&&&&\\
&0&&&&
\end{matrix}.$$

There are 5 Pl\"ucker coordinates vanishing on $Y_{\sigma'_3}$, they correspond to weights
$$\begin{matrix}1&0&0&0&0&0\\
&0&&&&\\
&0&&&&
\end{matrix},\qquad
\begin{matrix}-1&1&0&0&0&0\\
&0&&&&\\
&0&&&&
\end{matrix},\qquad
\begin{matrix}0&-1&1&0&0&0\\
&1&&&&\\
&0&&&&
\end{matrix},$$
$$\begin{matrix}0&0&1&0&0&0\\
&-1&&&&\\
&1&&&&
\end{matrix},\qquad
\begin{matrix}0&0&1&0&0&0\\
&0&&&&\\
&-1&&&&
\end{matrix}.$$

\begin{theorem}
  The minimal free resolution of $Y_{\sigma_3}$ is of the form:
\[
0\rightarrow A^2(-31)\rightarrow A^8(-21)\rightarrow A(-16)\oplus A^6(-15)\rightarrow A
  \]
  The minimal free resolution of $Y_{\sigma'_3}$ is of the form:
  \[
0\rightarrow A^4(-31)\rightarrow A^8(-25)\rightarrow A(-16)\oplus A^4(-15)\rightarrow A
  \]
  Furthermore, the ideals of these varieties are linked by a complete intersection of length $3$.

 The resolution for the sum of the two ideals has the form
  \[
0\rightarrow A(-46)\rightarrow A^8(-31)\oplus A(-30)\rightarrow A^8(-25)\oplus A^8(-21)\rightarrow A(-16)\oplus A^8(-15)\rightarrow A.
  \]
\end{theorem}

The numerator of the reduced Hilbert series of $\bk[Y_{\sigma_3}]$ is
$$2t^{28}+6t^{27}+12t^{26}+20t^{25}+30t^{24}+42t^{23}+56t^{22}+72t^{21}+90t^{20}+$$
  $$+110t^{19}+124t^{18}+132t^{17}+134t^{16}+130t^{15}+120t^{14}+105t^{13}+91t^{12}+78t^{11}+$$
  $$+66t^{10}+55t^9+45t^8+36t^7+28t^6+21t^5+15t^4+10t^3+6t^2+3t+1$$
  and hence its degree is 1640. The numerator of the reduced Hilbert series of $\bk[Y_{\sigma'_3}]$ is
$$4t^{28}+12t^{27}+24t^{26}+40t^{25}+60t^{24}+84t^{23}+104t^{22}+120t^{21}+132t^{20}+$$
  $$+140t^{19}+144t^{18}+144t^{17}+140t^{16}+132t^{15}+
  120t^{14}+105t^{13}+91t^{12}+78t^{11}+$$
  $$+66t^{10}+55t^9+45t^8+36t^7+28t^6+21t^5+15t^4+10t^3+6t^2+3t+1$$
  and hence its degree is $1960$.
  The numerator of the reduced Hilbert series of $\bk[Y_{\sigma_3} \cap Y_{\sigma'_3}]$ is
  \begin{align*}
t^{42}+4t^{41}+10t^{40}+20t^{39}+35t^{38}+56t^{37}+84t^{36}+
      120t^{35}+165t^{34}+220t^{33}+286t^{32}\\+364t^{31}+455t^{30}+
      560t^{29}+680t^{28}+808t^{27}+936t^{26}+1056t^{25}+1160t^{24}
    +1240t^{23}+1288t^{22}\\+1304t^{21}+1288t^{20}+
    1240t^{19}+1160
      t^{18}+1056t^{17}+936t^{16}+808t^{15}+680t^{14}+560t^{13}\\+455
    t^{12}+364t^{11}+286t^{10}+220t^{9}
    +165t^{8}+120t^{7}+84t^{6
      }+56t^{5}+35t^{4}+20t^{3}+10t^{2}+4t+1
  \end{align*}
  and its degree is $20400$.

\begin{remark}
We conjecture that the restriction of $Y_{\sigma_3}$ and $Y_{\sigma'_3}$ to $\bigwedge^3 \bk^8$ remain codimension 3. We have been unable to verify this even computationally; we now describe the restrictions assuming the fact.

There is an invariant $\Delta$ of degree $16$ on $\fg_1$, which can be described as a hyperdiscriminant of the representation $\bigwedge^3 \bk^8$, i.e., the equation of the hypersurface projectively dual to the Grassmannian (see for example  \cite[\S 9C]{JWm03}).

We write $\bk^8=F\oplus G$ with $\dim F=6$, $\dim G=2$ with $A' = \Sym(\fg_1)$. Conjecturally, when restricted to $A'$, the resolution of $Y_{\sigma_3}$ is
\[
  0\rightarrow A'^2(-31)\rightarrow A'^8(-21)\rightarrow A'(-16)\oplus A'^6(-15)\rightarrow A'.
\]
We also conjecture that the ideal of $Y_{\sigma_3}$ is generated by $\Delta$, $\partial\Delta /\partial{y_i}$, where $\lbrace y_i\rbrace$ is a basis of the dual of $F\otimes\bigwedge^2G$. As far as we are aware, this does not follow formally from the codimension conjecture in this case.

The equivariant format of this resolution is (in characteristic 0):
$$0\rightarrow \bS_{12,12,12,12,12,12}F\otimes \bS_{11,10}G\otimes A'(-31)\rightarrow
\begin{array}{c}
  \bS_{8,8,8,8,8,8}F\otimes \bS_{8,7}G\otimes A'(-9)\oplus\\
  \bS_{9,8,8,8,8,8}F\otimes \bS_{7,7}G\otimes A'(-21)
\end{array}
\rightarrow$$
\[
\begin{array}{c} \bS_{6,6,6,6,6,6}F\otimes \bS_{6,6}G\otimes A'(-16)\oplus\\
  \bS_{6,6,6,6,6,5}F\otimes  \bS_{5,5}G\otimes A'(-15)
\end{array}
\rightarrow A'.
\]

  To describe the restriction of $Y_{\sigma'_3}$, we set $\bk^8=F'\oplus G'$ where $\dim F'=4$, $\dim G'=4$ and $A''=\Sym(\fg_1)$. Conjecturally, when restricted to $A''$, the resolution is
  \[
    0\rightarrow A''^4(-31)\rightarrow A''^8(-25)\rightarrow A''(-16)\oplus A''^4(-15)\rightarrow A''.
\]
We conjecture that the ideal of $Y_{\sigma'_3}$ is generated by five elements  $\Delta$, $\partial\Delta /\partial{z_i}$, where $\lbrace z_i\rbrace$ is a basis of the dual of $\bigwedge^3G'$. Again, this does not seem to follow formally from the codimension conjecture.

The equivariant format of this resolution is (in characteristic 0):
$$0\rightarrow \bS_{13,12,12,12}F'\otimes \bS_{11,11,11,11}G'\otimes A''(-31)\rightarrow
\begin{array}{c}
  \bS_{10,10,10,10}F'\otimes \bS_{9,9,9,8}G'\otimes A''(-9)\oplus\\
  \bS_{10,10,10,9}F'\otimes \bS_{9,9,9,9}G'\otimes A''(-25)
\end{array}
\rightarrow$$
\[
  \begin{array}{c}
    \bS_{6,6,6,6}F'\otimes \bS_{6,6,6,6}G'\otimes A''(-16)\oplus\\
    \bS_{6,6,6,6}F'\otimes  \bS_{6,5,5,5}G'\otimes A''(-15)
  \end{array}
  \rightarrow A''. 
\]
\end{remark}

\section{Questions}

In the Dynkin case the most important problem is to decide whether the five resolutions we get for Dynkin types $\rE_6$, $\rE_7$, $\rE_8$ are generic resolutions for perfect ideals of codimension three with resolutions of these formats.
This is equivalent to saying that each perfect ideal of codimension 3 with a resolution of Dynkin format has a split complex ${\bf F}_\bullet^{\rm top}$.

The pattern with the Pl\"ucker coordinates and the pair of Schubert varieties of codimension 3 generalizes beyond Dynkin diagrams.
For any $\rT_{p,q,r}$ with $p=2$ we get in the homogeneous space $G(\rT_{p,q,r})/P_{x_1}$ ($G$ is a Kac--Moody group corresponding to $\rT_{p,q,r}$, $P_{x_1}$ is a parabolic corresponding to a simple root corresponding to vertex $x_1$) two opposite Schubert varieties $\Omega_{\sigma_3}$ and $\Omega_{\sigma'_3}$ of codimension 3. These are ind-varieties. 
The opposite Schubert varieties $\Omega_{\sigma_3}$ and $\Omega_{\sigma'_3}$ are normal and Cohen--Macaulay by \cite[Proposition 3.4]{KS09}.

In this case there is no analogue of the big cell $Y$.   Instead we should do the following. Let us denote by $p_w$ the Pl\"ucker coordinate corresponding to $w\in W/W_P$.
We have open sets $U_v$ ($v\in W/W_P$) in $G(\rT_{p,q,r})/P_{x_1}$ consisting of points for which $p_v\ne 0$. The sets $U_v$ are infinite dimensional affine spaces.
One should look at the sets $Y_{v,\sigma_3}:=\Omega_{\sigma_3}\cap U_v$ and $Y_{v,\sigma'_3}:=\Omega_{\sigma'_3}\cap U_v$.

In the non-Dynkin cases the main questions are as follows.
\begin{enumerate}
\item Can we find a sequence of open cells $Y_n$, cofinal in the Bruhat order on $W/W_P$, so that for each $n$, the defining ideals of $Y_n\cap Y_{\sigma_3}$ and of $Y_n\cap Y_{\sigma'_3}$ have the resolutions of the corresponding formats?
\item If the answer to the preceding question is yes, one could ask whether the series of resolutions ${\bf F}^{(n)}_\bullet$ of the cyclic modules given by the defining ideals of
$Y_n\cap Y_{\sigma_3}$ and of $Y_n\cap Y_{\sigma'_3}$ could have the versality property with respect to free resolutions of perfect ideals of this format, i.e., each such resolution comes by a change of rings from the resolution ${\bf F}^{(n)}_\bullet$ for some $n$.
\end{enumerate}

 One needs first to deal with the affine cases, i.e., the diagrams ${\hat \rE}_7=\rT_{2,4,4}$ (self-linked format $(1,6,8,3)$) and with ${\hat \rE}_8=\rT_{2,3,6}$ (formats $(1,8,9,2)$ and $(1,5,9,5)$).

\end{document}